\documentclass[11pt]{amsart}

\usepackage[T1]{fontenc}
\usepackage{lmodern}
\usepackage{microtype}
\usepackage{mathtools,amssymb}
\usepackage{enumitem}
\usepackage[hidelinks]{hyperref}
\hypersetup{
  pdftitle={Dittert's conjecture in dimension 16 via a joint-deficit scaling lemma},
  pdfauthor={Boris Kafidov},
  pdfsubject={A proof of Dittert's conjecture in dimension 16},
  pdfkeywords={Dittert conjecture, permanent, doubly stochastic matrix, doubly superstochastic matrix}
}

\numberwithin{equation}{section}
\setlist[enumerate]{leftmargin=*,itemsep=0.35em,topsep=0.45em}

\DeclareMathOperator{\permanent}{per}
\newcommand{\R}{\mathbb{R}}
\newcommand{\one}{\mathbf{1}}

\newtheorem{theorem}{Theorem}[section]
\newtheorem{lemma}[theorem]{Lemma}
\newtheorem{proposition}[theorem]{Proposition}
\newtheorem{corollary}[theorem]{Corollary}
\theoremstyle{remark}
\newtheorem{remark}[theorem]{Remark}

\title[Dittert's conjecture in dimension 16]
{Dittert's conjecture in dimension $16$\\
via a joint-deficit scaling lemma}

\author{Boris Kafidov}
\thanks{The proof strategy, central lemma, and most of the initial exposition were produced by OpenAI's GPT-5.6 Sol through ChatGPT under the direction of the author. See the disclosure at the end of the paper.}
\date{21 July 2026}


\begin{document}

\begin{abstract}
Let $K_n$ be the set of nonnegative $n\times n$ matrices whose entries sum to $n$, and let
\[
 \Phi(A)=\prod_{i=1}^n r_i+\prod_{j=1}^n c_j-\permanent(A),
\]
where $r_i$ and $c_j$ are the row and column sums of $A$.
We prove Dittert's conjecture in dimension $16$:
\[
 \Phi(A)\le 2-\frac{16!}{16^{16}},
\]
with equality only at the matrix all of whose entries are $1/16$.
The key observation is that the row-product and column-product deficits of a near maximizer share one permanent-deficit budget. This yields a sharper scalar dilation to a doubly superstochastic matrix and excludes boundary maximizers in dimension $16$.
Combined with Pang's recent preprint for $n\ge 17$, the result gives the conjecture for every $n\ge16$.
\end{abstract}

\maketitle

\section{Introduction}

Let
\[
 K_n=\left\{A=(a_{ij})\in\R_{\ge0}^{n\times n}:\sum_{i,j=1}^n a_{ij}=n\right\}.
\]
For $A\in K_n$, write
\[
 r_i=\sum_{j=1}^n a_{ij},\qquad c_j=\sum_{i=1}^n a_{ij},
\]
and define the Dittert functional
\[
 \Phi(A)=\prod_{i=1}^n r_i+\prod_{j=1}^n c_j-\permanent(A).
\]
Let $J_n=\one\one^{\mathsf T}$ and $U_n=n^{-1}J_n$. Dittert's conjecture asserts that
\[
 \Phi(A)\le 2-\gamma_n,
 \qquad \gamma_n:=\frac{n!}{n^n},
 \tag{1.1}\label{eq:dittert}
\]
for every $A\in K_n$, with equality only at $U_n$.

The cases $n=2$ and $n=3$ are classical. Pang's 2026 preprint proves the conjecture for every $n\ge17$ and explicitly leaves $4\le n\le16$ open \cite{Pang2026}. The purpose of this note is to close the endpoint $n=16$.

\begin{theorem}\label{thm:main}
For every $A\in K_{16}$,
\[
 \Phi(A)\le 2-\frac{16!}{16^{16}},
\]
with equality if and only if $A=U_{16}$.
\end{theorem}

The argument follows the boundary-exclusion framework of Cheon--Wanless and Pang, but retains a constraint discarded in the earlier dilation estimates: the deficits of the two products in $\Phi$ have a common total budget. This improves the dilation enough to reach dimension $16$.

At the time of writing, no earlier complete proof of the $n=16$ case was located, and \cite[Remark~2]{Pang2026} still records it as open. The priority claim here is therefore made only to the best of the author's knowledge.

\section{Preliminaries}

Let $\Omega_n$ denote the set of doubly stochastic $n\times n$ matrices. A nonnegative matrix $S$ is \emph{doubly superstochastic} if there exists $B\in\Omega_n$ such that $B\le S$ entrywise.
For index sets $I,J\subseteq\{1,\dots,n\}$, let $A[I,J]$ be the corresponding submatrix and let $s(A[I,J])$ denote the sum of its entries.

We use the following standard facts.

\begin{proposition}\label{prop:ingredients}
The following statements hold.
\begin{enumerate}[label=\textup{(\roman*)}]
 \item If $B\in\Omega_n$, then
 \[
   \permanent(B)\ge \gamma_n,
 \]
 with equality if and only if $B=U_n$ \cite{Egorychev1981,Falikman1981}.

 \item Every entrywise positive global maximizer of $\Phi$ on $K_n$ is $U_n$ \cite{Hwang1986}.

 \item A nonnegative $n\times n$ matrix $S$ is doubly superstochastic if and only if
 \[
   s(S[I,J])\ge |I|+|J|-n
   \tag{2.1}\label{eq:hall}
 \]
 for all $I,J\subseteq\{1,\dots,n\}$ \cite[Lemma~2.2]{CheonWanless2012}.

 \item If $n>3$, $B\in\Omega_n$, and $B$ has a zero entry, then, by the Knopp--Sinkhorn boundary estimate \cite{KnoppSinkhorn1982},
 \[
   \permanent(B)\ge m_n,
   \qquad
   m_n:=(n-2)!\left(\frac{n-2}{(n-1)^2}\right)^{n-2}.
   \tag{2.2}\label{eq:boundary}
 \]
\end{enumerate}
\end{proposition}

\section{The joint-deficit scaling lemma}

\begin{lemma}[Joint-deficit scaling]\label{lem:scaling}
Let $A\in K_n$ satisfy
\[
 \Phi(A)\ge 2-\gamma_n,
 \qquad
 \permanent(A)=\gamma_n-\delta,
 \qquad 0\le\delta<1.
\]
Set
\[
 t=\sqrt{\frac{n\delta}{1-\delta}}.
\]
If $t<1$, then $(1-t)^{-1}A$ is doubly superstochastic.
\end{lemma}

\begin{proof}
Put
\[
 R=\prod_{i=1}^n r_i=1-\rho,
 \qquad
 C=\prod_{j=1}^n c_j=1-\sigma.
\]
Since the row sums and column sums each have arithmetic mean $1$, AM--GM gives $R,C\le1$, hence $\rho,\sigma\ge0$. The hypothesis on $\Phi$ yields
\[
 (1-\rho)+(1-\sigma)-(\gamma_n-\delta)\ge2-\gamma_n,
\]
so
\[
 \rho+\sigma\le\delta.
 \tag{3.1}\label{eq:jointbudget}
\]
In particular $R,C>0$ and $0\le\rho,\sigma\le\delta<1$.

Fix a set $I$ of $k$ rows, where $1\le k\le n-1$, and write
\[
 \sum_{i\in I}r_i=k+\varepsilon.
\]
Applying AM--GM separately to the rows in $I$ and in its complement gives
\[
 R\le
 \left(1+\frac{\varepsilon}{k}\right)^k
 \left(1-\frac{\varepsilon}{n-k}\right)^{n-k}.
 \tag{3.2}\label{eq:twoblock}
\]
Let $p=k/n$ and $q=(k+\varepsilon)/n$. The logarithm of the right-hand side of \eqref{eq:twoblock} is $-nD(p\Vert q)$, where
\[
 D(p\Vert q)=p\log\frac pq+(1-p)\log\frac{1-p}{1-q}.
\]
The binary Pinsker inequality
\[
 D(p\Vert q)\ge2(p-q)^2
 \tag{3.3}\label{eq:pinsker}
\]
is elementary here: for fixed $p$, the derivative with respect to $q$ is $(q-p)/(q(1-q))$, and $q(1-q)\le1/4$; integration from $p$ to $q$ proves \eqref{eq:pinsker}. Consequently,
\[
 \log R\le-\frac{2\varepsilon^2}{n}.
\]
Since $R=1-\rho$,
\[
 |\varepsilon|
 \le \sqrt{-\frac n2\log(1-\rho)}
 \le \sqrt{\frac{n\rho}{2(1-\delta)}}.
 \tag{3.4}\label{eq:rowsubset}
\]
Here we used $-\log(1-x)\le x/(1-x)$ and $\rho\le\delta$. The same argument for columns gives, for every set $J$ of columns,
\[
 \left|\sum_{j\in J}c_j-|J|\right|
 \le \sqrt{\frac{n\sigma}{2(1-\delta)}}.
 \tag{3.5}\label{eq:columnsubset}
\]

Now take row and column sets $I,J$ with
\[
 p:=|I|+|J|-n>0,
\]
and put $a=|I^c|$ and $b=|J^c|$, so that $p=n-a-b$. By inclusion--exclusion and nonnegativity,
\begin{align*}
 s(A[I,J])
 &=n-\sum_{i\in I^c}r_i-\sum_{j\in J^c}c_j+s(A[I^c,J^c])\\
 &\ge p-
 \sqrt{\frac{n\rho}{2(1-\delta)}}-
 \sqrt{\frac{n\sigma}{2(1-\delta)}}\\
 &\ge p-\sqrt{\frac{n\delta}{1-\delta}}
 =p-t.
\end{align*}
The second inequality follows from
\[
 \sqrt\rho+\sqrt\sigma
 \le\sqrt{2(\rho+\sigma)}
 \le\sqrt{2\delta}
\]
and \eqref{eq:jointbudget}. Since $p$ is a positive integer, $p\ge1$, and therefore
\[
 p-t\ge(1-t)p.
\]
Thus
\[
 s\bigl((1-t)^{-1}A[I,J]\bigr)\ge p
\]
for every pair with $p>0$; pairs with $p\le0$ satisfy \eqref{eq:hall} automatically. Proposition~\ref{prop:ingredients}\textup{(iii)} now shows that $(1-t)^{-1}A$ is doubly superstochastic.
\end{proof}

\section{Boundary exclusion in dimension 16}

\begin{proof}[Proof of Theorem~\ref{thm:main}]
Because $K_{16}$ is compact and $\Phi$ is continuous, $\Phi$ has a global maximizer $A\in K_{16}$. Since $\Phi(U_{16})=2-\gamma$, where
\[
 \gamma:=\frac{16!}{16^{16}},
\]
we have $\Phi(A)\ge2-\gamma$.
If $A$ is positive, Proposition~\ref{prop:ingredients}\textup{(ii)} gives $A=U_{16}$. It remains to exclude a maximizer with a zero entry.

Write
\[
 \permanent(A)=\gamma-\delta.
\]
AM--GM gives $\Phi(A)\le2-\permanent(A)$, so maximality implies $\permanent(A)\le\gamma$; hence $\delta\ge0$. Since the permanent is nonnegative, $\delta\le\gamma$.
If $\delta=0$, then the two AM--GM product bounds must both be equalities, so $A\in\Omega_{16}$. Proposition~\ref{prop:ingredients}\textup{(i)} then forces $A=U_{16}$, contradicting the assumed zero. Thus
\[
 0<\delta\le\gamma.
 \tag{4.1}\label{eq:deltarange}
\]

Set
\[
 t=\sqrt{\frac{16\delta}{1-\delta}}.
\]
Since $\gamma=\prod_{k=1}^{16}(k/16)\le2^{-8}<1/17$, equation~\eqref{eq:deltarange} gives $t^2\le16\gamma/(1-\gamma)<1$. Lemma~\ref{lem:scaling} therefore implies that
\[
 S=(1-t)^{-1}A
\]
is doubly superstochastic. Hence some $B\in\Omega_{16}$ satisfies $B\le S$. Because $S$ has a zero entry, $B$ has a zero in the same position. With
\[
 m:=14!\left(\frac{14}{15^2}\right)^{14},
\]
Proposition~\ref{prop:ingredients}\textup{(iv)} and monotonicity of the permanent on nonnegative matrices give
\[
 \permanent(S)\ge\permanent(B)\ge m.
\]
Consequently,
\[
 \gamma-\delta
 =\permanent(A)
 =(1-t)^{16}\permanent(S)
 \ge m(1-t)^{16}.
 \tag{4.2}\label{eq:lowerper}
\]

We show that the reverse strict inequality holds. Bernoulli's inequality, \eqref{eq:deltarange}, and $x=\sqrt\delta$ give
\begin{align*}
 m(1-t)^{16}-(\gamma-\delta)
 &\ge m-\gamma+\delta-16m\sqrt{\frac{16\delta}{1-\delta}}\\
 &\ge m-\gamma+x^2-\frac{64m}{\sqrt{1-\gamma}}x\\
 &\ge m-\gamma-\frac{1024m^2}{1-\gamma}.
 \tag{4.3}\label{eq:quadratic}
\end{align*}
The last quantity is positive by exact rational comparison. Indeed,
\[
 \frac{1\,134\,226}{10^{12}}<\gamma<
 \frac{1\,134\,227}{10^{12}},
 \qquad
 \frac{1\,136\,699}{10^{12}}<m<
 \frac{1\,136\,700}{10^{12}}.
 \tag{4.4}\label{eq:rationalbounds}
\]
Therefore
\[
 m-\gamma>\frac{2472}{10^{12}},
\]
whereas
\[
 \frac{1024m^2}{1-\gamma}
 <\frac{1024(1\,136\,700)^2}
 {10^{12}(10^{12}-1\,134\,227)}
 <\frac{1324}{10^{12}}.
\]
Thus the right-hand side of \eqref{eq:quadratic} is greater than
\[
 \frac{1148}{10^{12}}>0.
\]
This contradicts \eqref{eq:lowerper}. Hence a global maximizer has no zero entry, and Proposition~\ref{prop:ingredients}\textup{(ii)} identifies it uniquely as $U_{16}$.
\end{proof}

\begin{corollary}\label{cor:range}
Together with Pang's result \cite{Pang2026}, Dittert's conjecture holds for every $n\ge16$. The dimensions $4\le n\le15$ are not settled by the present argument.
\end{corollary}

\begin{remark}
The gain over the dilation used in \cite{Pang2026} is precisely \eqref{eq:jointbudget}. Bounding the row and column errors separately loses the fact that their product deficits must sum to at most the single permanent deficit $\delta$. Retaining that joint constraint changes the dilation to
\[
 t=\sqrt{\frac{n\delta}{1-\delta}},
\]
which is sufficient at $n=16$.
\end{remark}

\section*{Disclosure of AI assistance}

The proof strategy, the joint-deficit scaling lemma, and most of the original proof text were produced by OpenAI's GPT-5.6 Sol through ChatGPT in response to prompts from Boris Kafidov. ChatGPT was also used to revise the exposition and prepare the LaTeX manuscript. Boris Kafidov selected the problem, directed the interactions and revisions, and is the sole named author; he assumes responsibility for the final manuscript and for any submission of it. The AI system is acknowledged as a reasoning and writing tool, not as an author. This disclosure is not a substitute for independent expert mathematical review.


\begin{thebibliography}{99}

\bibitem{CheonWanless2012}
G.-S. Cheon and I. M. Wanless,
\emph{Some results towards the Dittert conjecture on permanents},
Linear Algebra Appl. \textbf{436} (2012), 791--801.
\href{https://doi.org/10.1016/j.laa.2010.08.041}{doi:10.1016/j.laa.2010.08.041}.

\bibitem{Egorychev1981}
G. P. Egorychev,
\emph{The solution of van der Waerden's problem for permanents},
Soviet Math. Dokl. \textbf{23} (1981), 619--622.

\bibitem{Falikman1981}
D. I. Falikman,
\emph{Proof of the van der Waerden conjecture regarding the permanent of a doubly stochastic matrix},
Math. Notes Acad. Sci. USSR \textbf{29} (1981), 475--479.

\bibitem{Hwang1986}
S.-G. Hwang,
\emph{A note on a conjecture on permanents},
Linear Algebra Appl. \textbf{76} (1986), 31--44.
\href{https://doi.org/10.1016/0024-3795(86)90212-0}{doi:10.1016/0024-3795(86)90212-0}.

\bibitem{KnoppSinkhorn1982}
P. Knopp and R. Sinkhorn,
\emph{Minimum permanents of doubly stochastic matrices with at least one zero entry},
Linear Multilinear Algebra \textbf{11} (1982), 351--355.
\href{https://doi.org/10.1080/03081088208817459}{doi:10.1080/03081088208817459}.

\bibitem{Pang2026}
Z. Pang,
\emph{Proof of Dittert's conjecture for dimensions $n\ge17$},
arXiv:2606.01531v1 [math.RA], 1 June 2026.
\href{https://doi.org/10.48550/arXiv.2606.01531}{doi:10.48550/arXiv.2606.01531}.

\end{thebibliography}
\end{document}